\documentclass[12pt,oneside]{amsart}

\usepackage{amssymb,latexsym}

\usepackage{hyperref}
\hypersetup{
  colorlinks   = true, %Colours links instead of ugly boxes
  urlcolor     = blue, %Colour for external hyperlinks
  linkcolor    = blue, %Colour of internal links
  citecolor   = red %Colour of citations
}
\usepackage{anysize}
\marginsize{2cm}{2cm}{2cm}{2cm}

\usepackage[pdftex]{graphicx}

\newtheorem{theorem}{Theorem}[section]

\newtheorem{lemma}[theorem]{Lemma}

\newtheorem{conjecture}[theorem]{Conjecture}

\theoremstyle{definition}

\theoremstyle{remark}
\newtheorem{remark}[theorem]{Remark}
\newtheorem{question}[theorem]{Question}

\numberwithin{equation}{section}

\DeclareMathOperator{\vol}{vol}

\newcommand{\where}{\mathop{\ |\ }\nolimits}

\renewcommand{\epsilon}{\varepsilon}
\renewcommand{\phi}{\varphi}
\renewcommand{\kappa}{\varkappa}

\usepackage{ulem}
\usepackage[usenames]{color}
\usepackage{colortbl}

\begin{document}

\title{Viterbo's conjecture for certain Hamiltonians in classical mechanics}

\author{Roman~Karasev{$^{\spadesuit}$}}
\address{Roman Karasev, Dept. of Mathematics, Moscow Institute of Physics and Technology, Institutskiy per. 9, Dolgoprudny, Russia 141700}
\address{Roman Karasev, Institute for Information Transmission Problems RAS, Bolshoy Karetny per. 19, Moscow, Russia 127994}

\email{r\_n\_karasev@mail.ru}
\urladdr{http://www.rkarasev.ru/en/}

\author{Anastasia~Sharipova{$^{\clubsuit}$}}

\address{Anastasia Sharipova, Dept. of Mathematics, Moscow Institute of Physics and Technology, Institutskiy per. 9, Dolgoprudny, Russia 141700}
\address{Anastasia Sharipova, Pennsylvania State University Mathematics Dept., 411 McAllister Building, University Park, State College, PA 16802}

\email{independsharik@yandex.ru}

\thanks{{$^\spadesuit$}{$^\clubsuit$} Supported by the Russian Foundation for Basic Research grant 18-01-00036}
\thanks{{$^\spadesuit$} Supported by the Federal professorship program grant 1.456.2016/1.4}



\begin{abstract}
We study some particular cases of Viterbo's conjecture relating volumes of convex bodies and actions of closed characteristics on their boundaries, focusing on the case of a Hamiltonian of classical mechanical type, splitting into summands depending on the coordinates and the momentum separately. We manage to establish the conjecture for sublevel sets of convex $2$-homogeneous Hamiltonians of this kind in several particular cases. We also discuss open cases of this conjecture.
\end{abstract}

\maketitle

\section{Introduction}

In this note we consider the Viterbo conjecture~\cite{vit2000}, prove some of its new particular cases, and discuss relevant examples. We start with reminding the statement of the conjecture:

\begin{conjecture}[Viterbo, 2000]
For a convex body $X\subset\mathbb R^{2n}$ the following inequality holds
\[
\vol X \ge \frac{c_{EHZ}(X)^n}{n!}.
\]
\end{conjecture}

Here we treat the \emph{Ekeland--Hofer--Zehnder capacity} $c_{EHZ}(X)$ as the smallest action of a closed characteristic on the boundary $\partial X$, when $X$ is smooth, and extend it to non-smooth bodies by continuity (say, in the Hausdorff metric). A \emph{characteristic} is a curve $\gamma$ whose velocity $\dot\gamma$ is always in the one-dimensional kernel of the restriction of 
\[
\omega = \sum_{i=1}^n dp_i\wedge dq_i
\]
to the tangent space $T\partial X$. The \emph{action} of such a closed curve is the integral
\[
\int_\gamma \sum_{i=1}^n p_idq_i.
\]
For more details on these definitions and some results see~\cite{vit2000,aaok2013}.

Another point of view on this problem is to consider a proper (tending to $+\infty$ at infinity) convex Hamiltonian $H :\mathbb R^{2n} \to \mathbb R$ and compare the volume of a domain $X=\{H(p,q) \le E\}$ with the smallest action of a closed Hamiltonian trajectory on the boundary $\{H(p,q)=E\}$. Of course, this is just a reformulation, but we will mostly argue in these terms in this note, because we want to concentrate on certain types of Hamiltonians from classical mechanics.

So far Viterbo's conjecture seems rather hard, but one may try to attack its certain particular cases. As it was shown in \cite{aaok2013}, the case of $X = K\times T$ for two convex bodies $K\subset \mathbb R^n$ and $T\subset\mathbb R^n$ (Lagrangian product) is already interesting and for centrally symmetric $K$ and $T=K^\circ$ (the polar convex body) it is equivalent to the still open Mahler conjecture \cite{ma1939},
\[
\vol K\cdot \vol K^\circ \ge \frac{4^n}{n!}.
\]
In the case of the Lagrangian product the capacity $c_{EHZ}(K\times T)$ turns out to be the shortest length of the billiard trajectory in $K$ with length measured by the support function of $T$, or vice versa, see \cite{aao2012}. Nontrivial estimates of these capacities were made in \cite{aaok2013,abksh2014}.

In this work we consider another natural case, related to a Hamiltonian coming from classical mechanics with separated kinetic and potential energies, in the form
\[
H(p,q) = T(p) + V(q).
\]
Strictly speaking, this case already contains the case of Lagrangian products. Indeed, if we take two norms $\|p\|_K$ and $\|q\|_T$ for separate components of the canonical coordinates, and consider 
\[
H(p,q) = \|p\|_K^m + \|q\|_T^m
\]
then, for $m\to+\infty$, the domain $X=\{H(p,q) \le 1\}$ approaches the Lagrangian product of the unit balls of $\|\cdot\|_K$ in $p$ coordinates and $\|\cdot\|_T$ in $q$ coordinates. But in this note we mostly concentrate on the case $m=2$ in the above formula, when the potential and kinetic energies are convex and $2$-homogeneous.

Our results show the validity of Viterbo's conjecture for even $2$-homogeneous Hamiltonians with standard kinetic energy $T(p) = |p|^2/2$, for some other classical even $2$-homogeneous Hamiltonians in Theorems \ref{theorem:ellipsoid-in-between} and \ref{theorem:two-dim-squares}, and in a different situation, when the Hamiltonian is a sum of functions of the pairs $(p_i,q_i)$, Theorem \ref{theorem:sum-two-dim}.

\section{Viterbo's conjecture for some 2-homogeneous Hamiltonians}

Let us show how to handle some cases using the technique similar to~\cite[Section~2]{herm1998}. The first one, considered in this section, deals with the mechanical case and is a slight generalization of the harmonic oscillator, $H = \frac{1}{2}\sum_{i=1}^n p_i^2 + f_i^2 q_i^2$, for which the level surface is an ellipsoid and the capacity is known. The second one in Section \ref{section:in-pairs} has a separation of variables into pairs $(p_i,q_i)$ instead of separating them into $p$ and $q$, but is again a slight generalization of the harmonic oscillator.

In the next theorem we consider the case of the standard kinetic energy and an even $2$-homogeneous potential energy.

\begin{theorem}
\label{theorem:even-2-homogeneous}
Put $T=|p|^2/2$ (the standard kinetic energy, where $|\cdot|$ is the standard Euclidean norm) and let $V$ be proper,  even, and $2$-homogeneous. Then Viterbo's conjecture holds for the sets $\{T(p) + V(q) \le E\}$.
\end{theorem}
\begin{proof}
The assumption we have means $V(q) > 0$ for $q\neq 0$ and $V(tq) = t^2V(q)$ for every $t\in\mathbb R$. Write down the optimal inequality 
\[
V(q) \le \alpha |q|^2,
\]
with equality $V(q_0) = \alpha|q_0|^2$ for some nonzero vector $q_0$ and all its multiples (because of the homogeneity).

Consider 
\[
H(p,q) = \frac{1}{2}|p|^2 + V(q),\quad H'(p,q) = \frac{1}{2}|p|^2 + \alpha |q|^2
\]
and put (we may assume $E=1$ without loss of generality)
\[
X = \{H(p,q) \le 1\},\quad X' = \{H'(p,q) \le 1\}.
\]
It is clear that $X'\subseteq X$. Moreover, on the two-dimensional symplectic subspace $P\subset\mathbb R^{2n}$ spanned by $(0, q_0)$ and $(q_0, 0)$\footnote{Since we already use the Euclidean norm, we have an identification between the canonical coordinates $q$ and $p$.} the Hamiltonians $H$ and $H'$ coincide together with their derivatives (because the ratio $\frac{H}{H'}$ attains its maximum on this subspace). 

Since $C = P\cap \partial X' = P\cap \partial X$ is a closed trajectory for $H'$, then $C$ is also a closed trajectory for $H$. $X'$ is a symplectic ball, which is the equality case of Viterbo's conjecture. Since $X$ shares the same closed characteristic with $X'$, then $c_{EHZ}(X)\le c_{EHZ}(X')$. In fact, the convexity of $X$ and $X'$, and the monotonicity of the capacity implies $c_{EHZ}(X)=c_{EHZ}(X')$. On the other hand, the inclusion $X'\subseteq X$ obviously implies $\vol X \ge \vol X'$. Hence Viterbo's inequality is valid for $X$ as well.
\end{proof}

The previous theorem has a generalization. Before stating it we need to recall the standard and useful in classical mechanics notion of the Legendre transform for a lower semicontinuous (we actually only consider smooth) convex function $f : \mathbb R^n\to \mathbb R$:
\[
f^L(p) = \sup_{q\in\mathbb R^n} \left( \langle p, q\rangle - f(q) \right).
\]
It is known that $(f^L)^L = f$, assuming the function is allowed to take value $+\infty$. In case $f$ is $2$-homogeneous, $f^L$ also turns out to be $2$-homogeneous. Also note that the Legendre transform reverts the inequalities, $f\ge g\Rightarrow f^L\le g^L$.

\begin{theorem}
\label{theorem:ellipsoid-in-between}
Let $T, V : \mathbb R^n\to \mathbb R^+$ be proper, even, $2$-homogeneous and assume $T^L\ge C V$ with equality at some nonzero point $q_0$, and at the line through the origin spanned by $q_0$ from the homogeneity. Assume also that there exists a positive definite quadratic form $Q$ such that
\[
T^L \ge Q \ge C V.
\]
Then Viterbo's conjecture holds for the sets $\{T(p) + V(q) \le E\}$.
\end{theorem}
\begin{proof}
After a coordinate change we assume $Q=|q|^2/2$ and $C=1$ (the latter may scale the capacity and the volume accordingly). In this case $Q^L = |p|^2/2$, $T\le Q^L$ and $V\le Q$ with equality at $q_0$ and $p_0$ equal to $q_0$ in this coordinate system. Evidently, the circle 
\[
S = \{uq_0 + vp_0 : u^2 + v^2 = 1\}
\]
is a closed characteristic of the boundary of the ball
\[
X'=\left\{Q(q) + Q^L(p) = \frac{|q|^2 + |p|^2}{2} \le \frac{1}{2}\right\}.
\] 
Since on the linear span $\{uq_0+vp_0\}$ the derivatives of $T+V$ are the same as the derivatives of $Q^L+Q$ (because there are the critical points of the inequalities $T\le Q^L$ and $V\le Q$) then $S$ is also a closed characteristic of the boundary of $X=\{T(p)+V(q)\le 1/2\}$. From the inclusion $X'\subseteq X$ and the monotonicity $c_{EHZ}(X')\le c_{EHZ}(X)$ it follows that $c_{EHZ}(X_0) = c_{EHZ}(X) = \pi$ and also
\[
\vol(X) \ge \vol(X') = \frac{\pi^n}{n!}.
\]
\end{proof}

\begin{remark}
Note that in the previous results we establish the Viterbo inequality for $X$ by putting a symplectic ball into $X$ so that $\partial X$ and the boundary of the ball share a common closed characteristic. This proves that their Ekeland--Hofer--Zehnder capacities are equal, while the inequality for volumes follows from the inclusion. It is still an open problem (the strong version of Viterbo's conjecture) to prove or disprove that it is possible to put (an open) symplectic ball of capacity $C$ into any convex body $X\subset\mathbb R^{2n}$ of Ekeland--Hofer--Zehnder capacity $C$.
\end{remark}

\section{Viterbo's conjecture for a sum of Hamiltonians of pairs of canonical variables}
\label{section:in-pairs}

Here we give an example of another kind, when we consider the sum of one-dimensional motions.

\begin{theorem}
\label{theorem:sum-two-dim}
Viterbo's conjecture holds for the sets $\{H\le E\}$ when we deal with a direct sum of one-dimensional motions
\[
H(p, q) = \sum_{i=1}^n H_i(p_i, q_i),
\]
each $H_i : \mathbb R^2\to\mathbb R^+$ being proper with unique minimum and no other critical points in the plane.
\end{theorem}

\begin{proof}
This is essentially the result of~\cite[Section 2]{herm1998}, but we provide the explanation here for completeness.

In this case the trajectories are characterized as curves whose every projection to the space of $(p_i,q_i)$ is a trajectory of $H_i$. For a closed trajectory, let the corresponding area of the projection be $A_i = \int_\gamma p_idq_i$. Since the given Hamiltonians $H_i$ are proper with unique minimum at the origin, for every $H_i$ we find a unique closed trajectory with given $A_i$, and denote the value of $H_i$ on such a trajectory by $E_i(A_i)$. Such closed trajectories in coordinates $(p_i,q_i)$ will produce a closed trajectory in all the coordinates if the periods coincide, which translates to the fact that the numbers $\frac{\partial E_i}{\partial A_i}$ are equal to the same number for all $i$ such that $A_i > 0$. Note that this condition is sufficient, but not necessary, in the general case the periods of the projections are rational multiples of each other.

Note that every $E_i$ is a strictly increasing function of $A_i$. Consider the $A_i$ as the coordinates and study the function $A = A_1+\dots+A_n$ on the set
\[
Y_E = \{(A_1,\ldots, A_n) \where E_1(A_1) + \dots + E_n(A_n) = E\}.
\]
This set is diffeomorphic to a simplex of dimension $n-1$, and the critical points of $A$ may belong to its faces, which without loss of generality are given by 
\[
A_{k+1} = \dots = A_n = 0.
\]
For the remaining $A_1,\ldots,A_k$ we have from the Lagrange multiplier method:
\[
\frac{\partial E_1}{\partial A_1} = \dots = \frac{\partial E_k}{\partial A_k} = L.
\]
This precisely means that any critical value $A$ on $Y_E$ corresponds to a closed trajectory with action $A$. Let $A_E$ be the minimal of the actions. It is clear that the set
\[
X_E = \{(A_1,\ldots, A_n) \where E_1(A_1) + \dots + E_n(A_n) \le E\}
\]
contains the simplex $A_1+\dots+A_n < A_E$, the latter evidently having the symplectic volume $\frac{A_E^n}{n!}$. Hence the volume of $X_E$ is at least $\frac{A_E^n}{n!}$.

Consider $A_i : \mathbb R^2\to\mathbb R^+$ as a function, mapping a point $(p,q)$ to the action of the $H_i$-Hamiltonian closed trajectory passing through $(p,q)$. They together constitute a map $A : \mathbb R^{2n}\to (\mathbb R^+)^n$, evidently preserving the volume form, hence for the preimage $X'_E = A^{-1}(X_E)$ we have 
\[
\vol X'_E = \vol X_E = \frac{A_E^n}{n!}.
\]
Also $c_{EHZ}(X'_E)\le A_E$, because the critical point of $A$ we have found represents a closed characteristic on the boundary of $X'_E$, which is a sublevel set of the original Hamiltonian.
\end{proof}

\begin{remark}
Theorems \ref{theorem:even-2-homogeneous} and \ref{theorem:sum-two-dim} need no convexity assumption if we define $c_{EHZ}$ of the set $\{H \le E\}$ as the smallest action of a closed characteristic on its boundary. Note that thus defined value should not be called a \emph{symplectic capacity} since the capacities are assumed to be inclusion preserving.
\end{remark}

\section{Further questions and examples}

\subsection{General questions}

Going back to the case when $p_i$ and $q_i$ are separated, we note that the other small steps towards the general Viterbo conjecture would be:

\begin{question}
\label{question:2-hom-sum}
What if $H(p,q) = T(p) + V(q)$ with arbitrary proper, convex, even $2$-homogeneous $T$ and $V$?
\end{question}

\begin{question}
What if $H(p,q) = |p|^2/2 + V(q)$ with arbitrary proper convex $V$, not even, or not $2$-homogeneous?
\end{question}

Let us discuss Question~\ref{question:2-hom-sum} and its particular cases a bit more. Similar to the result about Lagrangian products in~\cite{aaok2013}, we expect the following case to be interesting:
\[
H(p,q) = \|p\|_*^2 +\|q\|^2,
\]
for a pair of arbitrary (symmetric) norm and its dual. In this case the convex body 
\[
X = \{(p,q) \where \|p\|_*^2 +\|q\|^2 \le 1\}
\]
is called the \emph{$L_2$-sum} $K^\circ\oplus_2 K$, where $K$ is the unit ball of $\|\cdot \|$ and $K^\circ$ is its polar, the unit ball of $\|\cdot\|_*$. 

In order to find closed trajectories of Hamiltonians of this kind, we can choose any direction with $q_0\in\partial K$ and its corresponding momentum $p_0 = \frac{\partial \|q\|}{\partial q} |_{q=q_0}\in \partial K^\circ$. Obviously, a closed Hamiltonian trajectory on the boundary of $X$, lying in the 2-plane spanned by $q_0$ and $p_0$, is possible for the $2$-homogeneity. This trajectory behaves like a harmonic oscillator and its action is evidently  equal to $\pi$. Could it happen that such trajectories have the smallest action of all trajectories on the boundary of $X$? Assuming Viterbo's conjecture we would then obtain
\[
\vol X \ge \frac{\pi^n}{n!}.
\]
Note that for any norm $\|\cdot\|$ on $\mathbb R^m$
\[
\int_{\mathbb R^m} e^{-\|x\|^\alpha}\; dx = \vol \{x \where \|x\| \le 1\}\cdot \Gamma(m/\alpha+1),
\]
and, therefore, for the $L_2$-sum in $\mathbb R^n\times \mathbb R^n$, we take $\alpha=2$ and split the integral of $e^{-\|p\|_*^2 - \|q\|^2}$ to two integrals to obtain (the binomial coefficient is defined for non-integer values using the $\Gamma$-function)
\begin{equation}
\label{equation:ell-two-sum}
\binom{n}{n/2} \vol K^\circ\oplus_2 K = \vol K^\circ \cdot \vol K.
\end{equation}
Hence our assumptions read:
\[
\vol K^\circ \cdot \vol K \ge \binom{n}{n/2} \frac{\pi^n}{n!},
\]
or, after a simplification,
\[
\vol K^\circ \cdot \vol K \ge \frac{\pi^n}{((n/2)!)^2}.
\]
But the latter is the Blaschke--Santal\'o inequality~\cite{bla1917,san1949} in the opposite direction, which is generally false. 

Hence, the Hamiltonians of the form $\|p\|_*^2 +\|q\|^2$ with non-Euclidean norms seem to have closed trajectories ``faster'' than the back-and-forth oscillations along a straight line through the origin. At this point, we get convinced that it makes sense to investigate the basic case (the $\ell_1$ and $\ell_\infty$ norms)
\[ 
H(p, q) = \|p\|_1^2 + \|q\|_\infty^2,
\]
corresponding to the known equality case in Mahler's conjecture, and examine its closed trajectories more closely. We do this below. 

\subsection{A two-dimensional example}

We now consider the simplest case when $n=2$ and the symplectic space is $\mathbb{R}^4$. In addition to one-dimensional back-and-forth trajectories mentioned above, it is easy to see from the canonical Hamiltonian equations 
\begin{eqnarray}
\dot{q_i} &=& \frac{\partial H}{\partial p_i} = 2(|p_1| + |p_2|)\frac{\partial |p_i|}{\partial p_i}\\\dot{p_i} &=& -\frac{\partial H}{\partial q_i} = -2\max\limits_{j}|q_j|\frac{\partial \max_{j}|q_j|}{\partial q_i}
\end{eqnarray} 
that any other trajectory looks like the one in Figure~\ref{figure:1} (its projection onto $p$- and $q$-subspace are shown) or this one with the opposite direction obtained by the substitution $t \rightarrow -t$.

\begin{figure}[ht]\center
\begin{tabular}{cc}
\includegraphics[width=40mm]{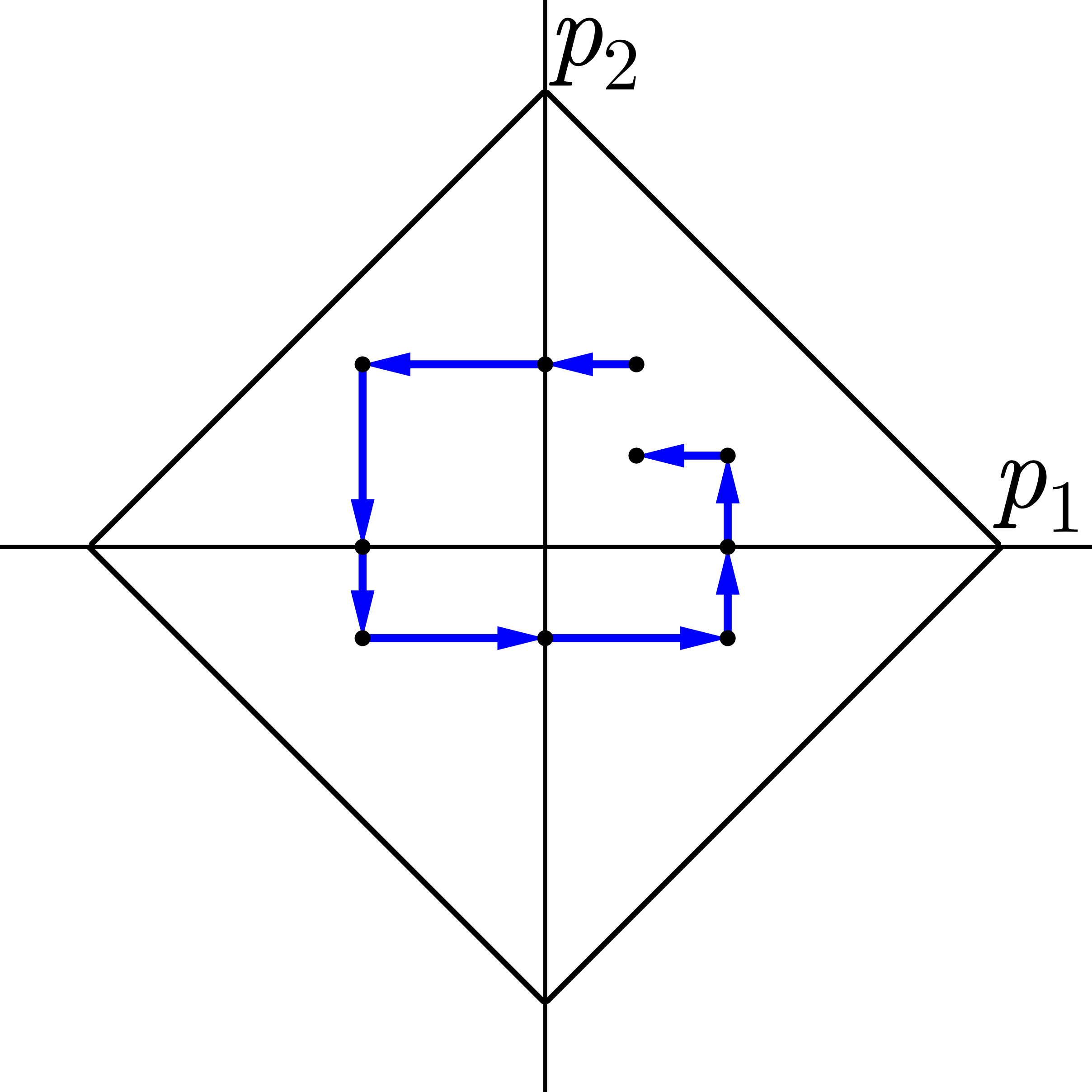}
&
\includegraphics[width=40mm]{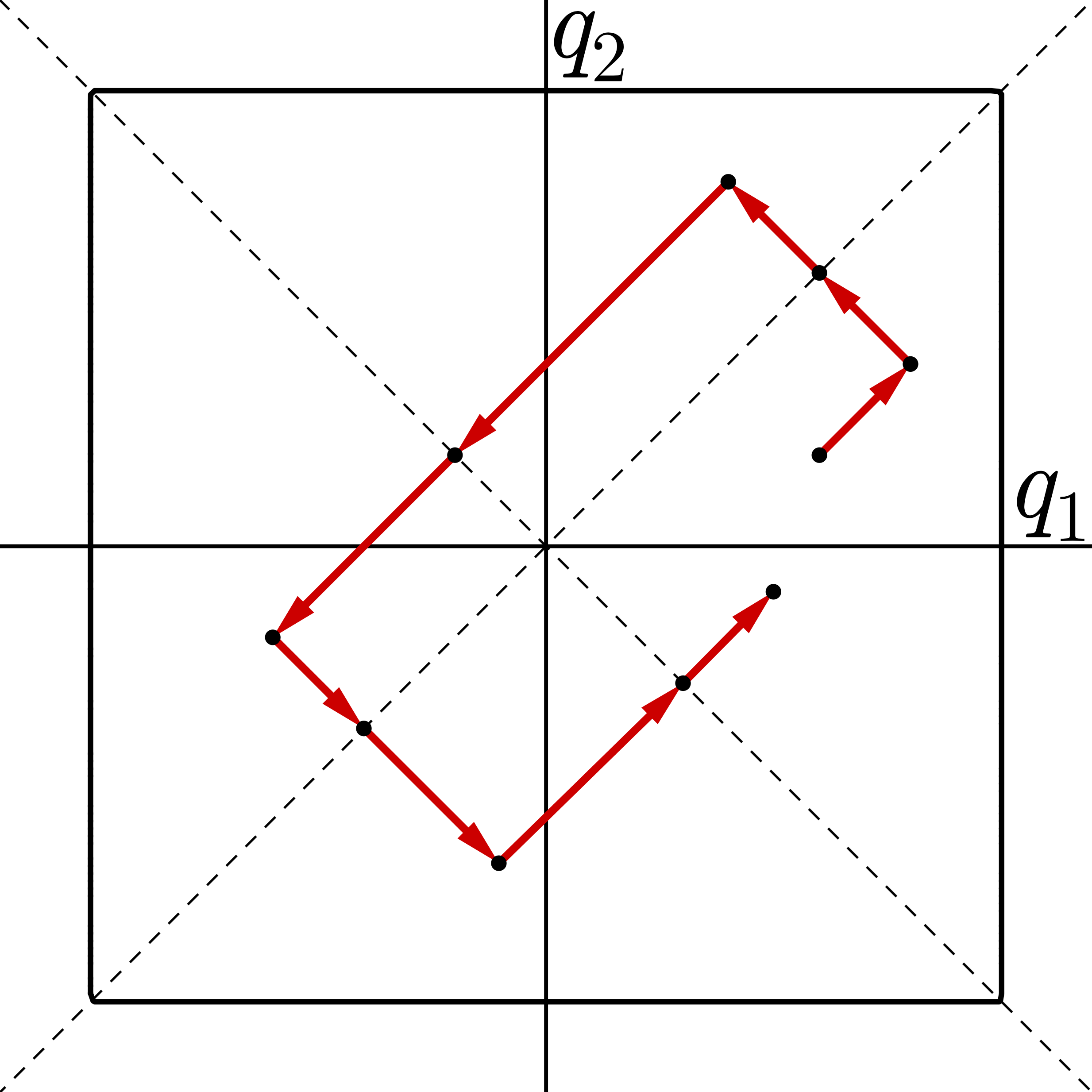}
\end{tabular}
\caption{A trajectory}
\label{figure:1}
\end{figure}

Take a closer look at a trajectory's cycle. In the planar case we can rotate one of the planes to make both norms look similar (although making a non-symplectic transform), that is we introduce variables 
\[
x_1 = \frac{q_1 - q_2}{2},\quad x_2 = \frac{q_1 + q_2}{2},
\] 
for which $\|q\|_{\infty} = \|x\|_1$, $H = \|p\|_1^2 + \|x\|_1^2$, and write the new equations of motion 
\[
\dot{p}_{1,2} = -\|x\|_1(\text{sgn}x_2 \pm \text{sgn}x_1),\quad \dot{x}_{1,2} = \|p\|_1(\text{sgn}p_1 \mp \text{sgn}p_2),
\] 
whose trajectory is illustrated in Figure~\ref{figure:2}. Consider the cycle of trajectory's points corresponding to non-smooth switches in the Hamiltonian: 
\begin{multline*}
(0, p_2, x_1, x_2) \rightarrow (p_1, p_2, 0, x_2) \rightarrow (p_1, 0, x_1', x_2) \rightarrow (p_1, p_2', x_1', 0) \rightarrow (0, p_2', x_1', x_2') \rightarrow\\
 \rightarrow (p_1', p_2', 0, x_2') \rightarrow (p_1', 0, x_1'', x_2') \rightarrow (p_1', p_2'', x_1'', 0) \rightarrow (0, p_2'', x_1'', x_2''). 
\end{multline*}
Here we assume positive variables for the first point: $p_2, x_1, x_2 > 0$. Such a switch occurs when one of the variables becomes zero, and we call it a \emph{turning point}. Let us also call the path between neighboring turning points a \emph{line segment}. Note that any line segment has two constant variables and we can identify the remaining variable in turning points from the expression:
\begin{equation}
    \label{surf}
    \|p\|_1^2 + \|x\|_1^2 = (|p_1| + |p_2|)^2 + (|x_1| + |x_2|)^2 = 1.
\end{equation} 

Define the sequence of absolute values of new variables appearing at the end of each line segment during the motion as $\{a_i\}$, and sequences of acute angles 
\[
\alpha_i = \arccos{a_{2i - 1}},\quad \beta_i = \arcsin{a_{2i}} \in [0, \frac{\pi}{2}].
\] 
For example, for already considered cycle 
\[
a_1 = |p_2| = \cos{\alpha_1},\ a_2 = |x_2| = \sin{\beta_1},\ a_3 = |p_1| = \cos{\alpha_2},\ a_4 = |x_1'| = \sin{\beta_2},\ a_5 = |p_2'|,
\] 
and so on. Note that angles $\alpha_i$ are used for $p$ and $\beta_i$ are used for $x$ (former $q$) coordinates. This sequence satisfies a recurrence 
\begin{equation}
    \label{rec}
    a_{n+1} = \sqrt{1-a_n^2} - a_{n-1}.
\end{equation}      

%Рисунок 2

\begin{figure}[ht]\center
\begin{tabular}{cc}
\includegraphics[width=40mm]{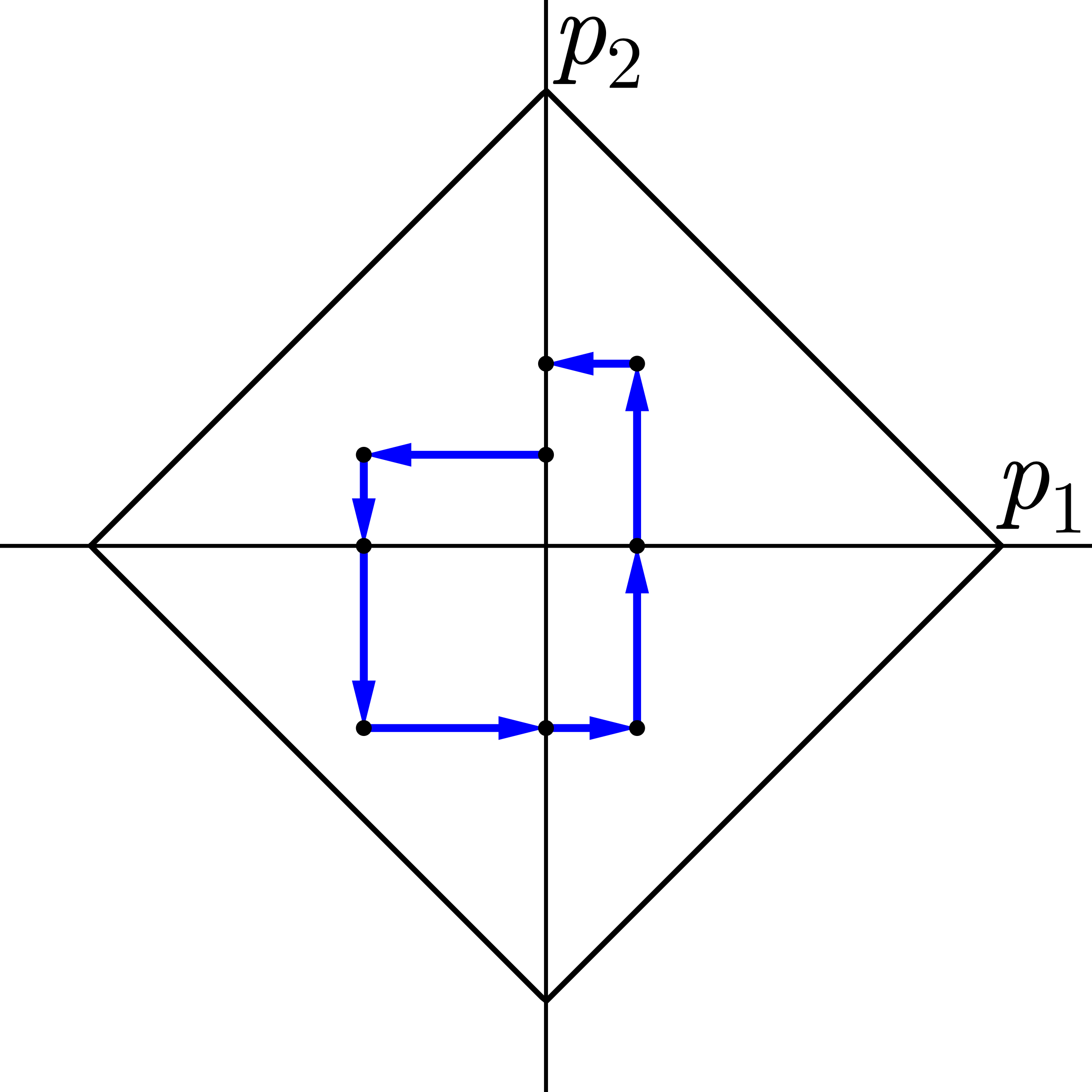}
&
\includegraphics[width=40mm]{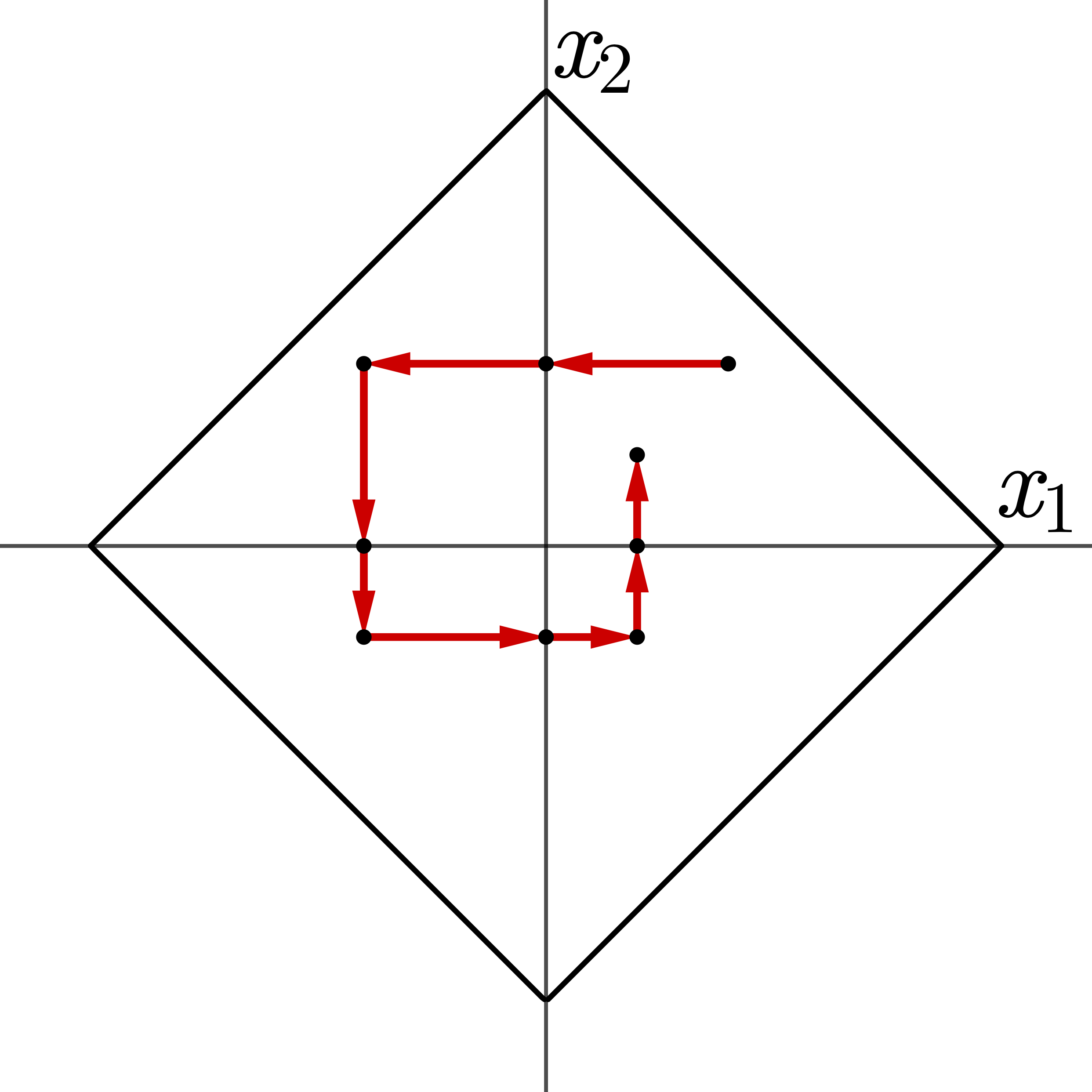}
\end{tabular}
\caption{One cycle of a trajectory}
\label{figure:2}
\end{figure}

Such an approach gives us a very simple expression for the action of a closed trajectory which consists of several cycles. It is known that for $2$-homogeneous Hamiltonians the action of a closed trajectory is equal to its period in terms of time for the level set $\{H=1\}$. Since our Hamiltonian is $2$-homogeneous, we just find the time. As already mentioned, any line segment has two constant variables $a_i$, $a_{i+1}$ expressed by angles $\alpha$ and $\beta$, and therefore the time required to pass this segment is also determined by them. 

By solving differential equations in one particular case, we find the time of motion from point $A = (0, p_2, x_1, x_2)$ to point $B = (p_1, p_2, 0, x_2)$ with $p_1 < 0, p_2, x_1, x_2 > 0$:
\begin{eqnarray*}
 \dot{p}_1(t) &=& -2x_1(t) - 2x_2(t),\\
 \dot{p}_2(t) &=& 0,\\
 \dot{x}_1(t) &=& 2p_1(t) - 2p_2(t),\\
 \dot{x}_2(t) &=& 0
\end{eqnarray*}
with the initial condition 
\begin{eqnarray*}
 p_1(0) &=& 0,\\
 x_1(0) &=& x_1.
\end{eqnarray*}
The solution appears to be
\begin{eqnarray*}
 p_1(t) &=& -(x_1 +x_2)\sin{2t} - p_2\cos{2t} + p_2,\\
 x_1(t) &=& (x_1 +x_2)\cos{2t} - p_2\sin{2t} -x_2,
\end{eqnarray*}
and we find $t_1$ from 
\begin{eqnarray*}
 p_1(t_1) &=& p_1,
 \\x_1(t_1) &=& 0,
\end{eqnarray*}
as
\[
t_1 = \frac{1}{2}\arcsin\left(\sqrt{1-x_2^2}\sqrt{1-p_2^2} - p_2x_2\right) = \frac{1}{2}(\alpha_1 - \beta_1).
\]

In the general case, for a line segment with constant $a_i$, $a_{i+1}$ expressed by angles $\alpha$ and $\beta$ ($\alpha \ge \beta$) the time required to pass $t = \frac{1}{2}(\alpha - \beta)$. To calculate the time of the whole trajectory we have to sum the times of all line segments. It is difficult to obtain an exact closed expression for the action, but it is possible to estimate it.

\begin{lemma}
\label{lemma:cycle}
Let $\gamma$ be a closed trajectory of $H = \|p\|_1^2 + \|q\|_{\infty}^2 = 1$ with $k$ cycles. Then the action of $\gamma$ is at least $k$.
\end{lemma}

\begin{proof}
For $\gamma$ we have a defining sequence 
\[
{\cos\alpha_1,\sin\beta_1, \cos\alpha_2, \ldots , \sin\beta_n},
\]
where $2n = 8k$ as one cycle generates eight angles, and the constraints from \eqref{rec} 
\[
\cos\alpha_i + \cos\alpha_{i + 1} = \cos\beta_i,\quad \sin\beta_{i - 1} + \sin\beta_i = \sin\alpha_i.
\]
Summing all the equations we obtain
\[
2\sum_{i=1}^n\cos{\alpha_i} = \sum_{i=1}^n\cos{\beta_i}\quad \text{and}\quad 2\sum_{i=1}^n\sin{\beta_i} = \sum_{i=1}^n\sin{\alpha_i}.
\]
Then for the action we have
\[
S = \frac{1}{2}(\alpha_1 -\beta_1 + \alpha_2 - \beta_1 + \alpha_2 - \beta_2 + \dots + \alpha_1 - \beta_n),
\]
and therefore
\begin{multline*}
2S \ge \cos\beta_1 - \cos\alpha_1 + \sin\alpha_2 - \sin\beta_1 + \cos\beta_2 - \cos\alpha_2 + \dots + \sin\alpha_1 - \sin\beta_n =\\
= \sum_{i=1}^n\cos\beta_i - \sum_{i=1}^n\cos\alpha_i + \sum_{i=1}^n\sin\alpha_i - \sum_{i=1}^n\sin\beta_i = \sum_{i=1}^n\cos\alpha_i + \frac{1}{2}\sum_{i=1}^n\sin\alpha_i \ge \frac{n}{2}.
\end{multline*}
Since the function $\cos{\alpha} + \frac{1}{2}\sin{\alpha}$ has the minimal value $\frac{1}{2}$ on $[0,\frac{\pi}{2}]$, we obtain $S \ge \frac{n}{4} = k$.
\end{proof}

\begin{theorem}
\label{theorem:two-dim-squares}
The smallest action of a closed trajectory of $H = \|p\|_1^2 + \|q\|_{\infty}^2 = 1$ in $\mathbb R^4$ is $4\arcsin{\frac{3}{5}}$ and Viterbo's conjecture holds true for its sublevel sets, the $\ell_2$-sums of the square and its polar, also a square.
\end{theorem}

\begin{proof}
As it was shown in \cite{akopyan2018capacity}, a centrally symmetric Hamiltonian should have a centrally symmetric minimal closed trajectory. Any such trajectory should consists of odd number of cycles as it may be viewed as an odd map from the circle to $\mathbb R^4$. Thus it is sufficient to consider centrally symmetric trajectories with $1$, $3$, $5$, and so on cycles around the origin.

In case of one cycle we have a defining sequence $\{a_i\}$ with $a_{i + 4} = a_i$ and obtain from \eqref{rec} 
\[
a_2^2 + (a_1 + a_3)^2 = 1 = (a_5 + a_3)^2 + a_4^2,\quad a_2 = a_4,
\]
and similarly $a_1 = a_3$. Therefore for any $n \neq m$ we have 
\[
a_{2n+1} = a_{2m+1} = a,\quad a_{2n} = a_{2m} = b.
\] 
Again using \eqref{rec} or \eqref{surf}: $a^2 + 4b^2 = 1, b^2 + 4a^2 = 1$. This gives 
\[
a_i^2 = \frac{1}{5},\quad \cos{\alpha_i} = \sin{\beta_i} = \frac{1}{\sqrt{5}},\quad  \sin{\alpha_i} = \cos{\beta_i} = \frac{2}{\sqrt{5}},
\] 
\[
\sin(\alpha_i - \beta_i) = \sin{\alpha_i}\cos{\beta_i} - \cos{\alpha_i}\sin{\beta_i} = \frac{2}{\sqrt{5}}\cdot\frac{2}{\sqrt{5}} - \frac{1}{\sqrt{5}}\cdot\frac{1}{\sqrt{5}} = \frac{3}{5},
\]
\[
\alpha_i - \beta_i = \arcsin{\frac{3}{5}},
\] 
and 
\[
S = \sum_{i=1}^4\alpha_i - \sum_{i=1}^4\beta_i = 4\arcsin{\frac{3}{5}} = 2.57\ldots
\]

If the trajectory has more than one cycle then it has at least 3 cycles and its action is at least 3 from Lemma \ref{lemma:cycle}. A back-and-forth trajectory has the action $S = \pi$. Thus we have found a trajectory with the smallest action.

It remains to check Viterbo's conjecture for this minimal trajectory. Indeed, for the body $X = \{H = 1\}$ we have the Ekeland--Hofer--Zehnder capacity
\[
c_{EHZ}(X) = S = 4\arcsin{\frac{3}{5}} = 2.5740\ldots < 2\sqrt2
\]
and 
\[
\text{vol}(X) = 4 = \frac{(2\sqrt2)^2}{2} > \frac{c(X)^2}{2}.
\]
\end{proof}

Note that the trajectory found in $\mathbb R^4$ also gives the estimate for the case $\mathbb{R}^6$, by considering the trajectory placed in the coordinate subspace $\mathbb R^4\subset\mathbb R^6$. Using the formula that we have already mentioned,
\[
\text{vol}(K^{\circ} \oplus_2 K) = \frac{\left((n/2)!\right)^2}{\left(n!\right)^2}\cdot 4^n,
\]
for $n = 3$ we obtain
\[
\text{vol}(K^{\circ} \oplus_2 K) = \pi > \frac{(4\arcsin{\frac{3}{5}})^3}{3!} \ge \frac{c^3}{3!}.
\]
Hence, Viterbo's conjecture holds for $n=3$ as well. However, for higher dimensions such trajectories are insufficient to show the validity of Viterbo's conjecture. In fact, for the capacity (action) we have inequality
\[
c(K^{\circ} \oplus_2 K) \le 4\cdot\left(\frac{((\frac{n}{2})!)^2}{n!} \right)^{\frac{1}{n}},
\]
hence for $n = 4$, 
\[
c(K^{\circ} \oplus_2 K) \le 4\cdot\left(\frac{1}{6}\right)^{\frac{1}{4}} = 2.55\ldots < 2.57\ldots
\]

\subsection{A higher-dimensional example}

We are going to generalize the above two-dimensional example to dimension $n>2$. We will not be able to determine the shortest closed trajectory, but will prove that Viterbo's conjecture holds in this case. 

\begin{theorem}
Viterbo's conjecture holds for the convex body in $\mathbb R^{2n}$ defined by
\[
H(p,q) = \|p\|^2_1 + \|q\|^2_{\infty} \le 1.
\]
\end{theorem}

\begin{proof}
Consider the trajectory on the surface $\{H(p,q)=1\}$ starting at the point with coordinates
\begin{eqnarray*}
q_i(0) &=& \frac{n - 2(i-1)}{\sqrt{n^2 + (n-1)^2}},\quad i = 1,\ldots,n\\
p_1(0) &=& 0,\ p_j(0) = \frac{1}{\sqrt{n^2 + (n-1)^2}},\quad j = 2,\ldots,n.
\end{eqnarray*}
Let us write the Hamiltonian equations for this Hamiltonian:
\begin{eqnarray*}
\dot{q_i} &=& \frac{\partial H}{\partial p_i} = 2(|p_1| + \dots + |p_n|)\frac{\partial |p_i|}{\partial p_i},\\
\dot{p_i} &=& -\frac{\partial H}{\partial q_i} = -2\max\limits_{j}|q_j|\frac{\partial \max_{j}|q_j|}{\partial q_i}.
\end{eqnarray*}
If $|q_k| = \max\limits_{j}|q_j|$ then we have
\begin{eqnarray*}
\dot{q_i} &=& 2(|p_1| + \dots + |p_n|)\cdot \text{sgn}p_i,\\
\dot{p_k} &=& -2q_k,\ \dot{p_j} = 0\quad \text{for}\quad j \not= k.
\end{eqnarray*}
In our case $q_1(0) = \max\limits_{j}|q_j|$ and $p_1(0) = 0$, so after the initial moment of time $q_1$ decreases and $q_j$ increases with the same velocity, $p_1$ decreases and $p_j$ does not change. Such motion continues until the moment $t_1$, when $q_1$ stops being the largest in the absolute value and gets equal to $q_2$. At this moment we have
\begin{eqnarray*}
q_1(t_1) &=& \frac{n - 1}{\sqrt{n^2 + (n-1)^2}} = q_2(t_1),\ q_i(t_1) = \frac{n - 2(i-1) + 1}{\sqrt{n^2 + (n-1)^2}}, i = 3,\ldots,n\\
p_1(t_1) &=& \frac{- 1}{\sqrt{n^2 + (n-1)^2}},\ p_j(t_1) = \frac{1}{\sqrt{n^2 + (n-1)^2}}, j = 2,\ldots,n.
\end{eqnarray*}
Then the coordinates $q_1,q_2,\ldots, q_n$ become the largest one by one, arriving to $\frac{n}{\sqrt{n^2 + (n-1)^2}}$ before $q_1$ reaches the value $\frac{- n + 1}{\sqrt{n^2 + (n-1)^2}}$ at the moment $t_{2n}$ and becomes the largest in the absolute value again going to $q_1(t_{2n+1}) = \frac{- n}{\sqrt{n^2 + (n-1)^2}}$. Thereafter the trajectory goes in the centrally symmetric way until the moment $t_{4n}$, when the trajectory returns to the starting point. The projections of the trajectory to $p$ and $q$ spaces for $n = 3$ are shown in Figure \ref{Figure 3}. The graphs for the evolution of the $q_i$ and the $p_i$ in time are shown in Figures \ref{Figure 4}.

% Рисунок 3

\begin{figure}[ht]\center
\begin{tabular}{cc}
\includegraphics[width=80mm]{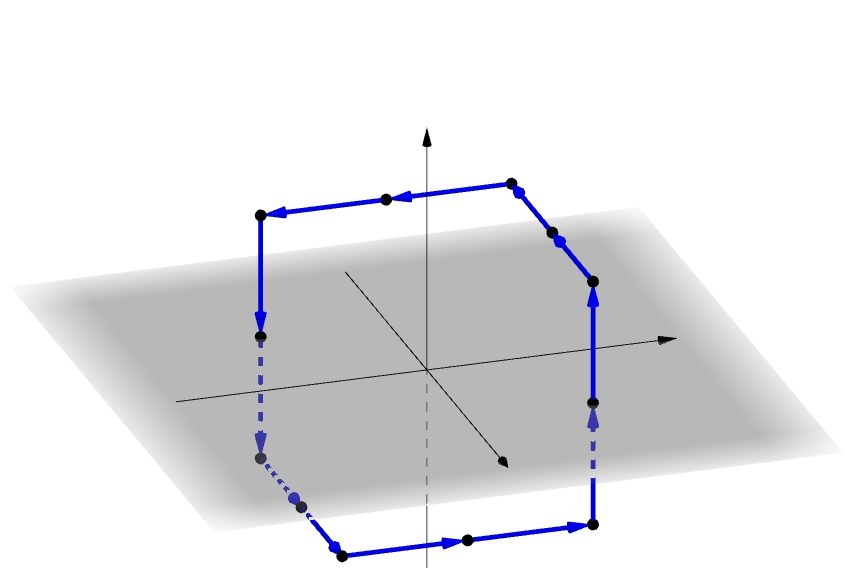}
&
\includegraphics[width=80mm]{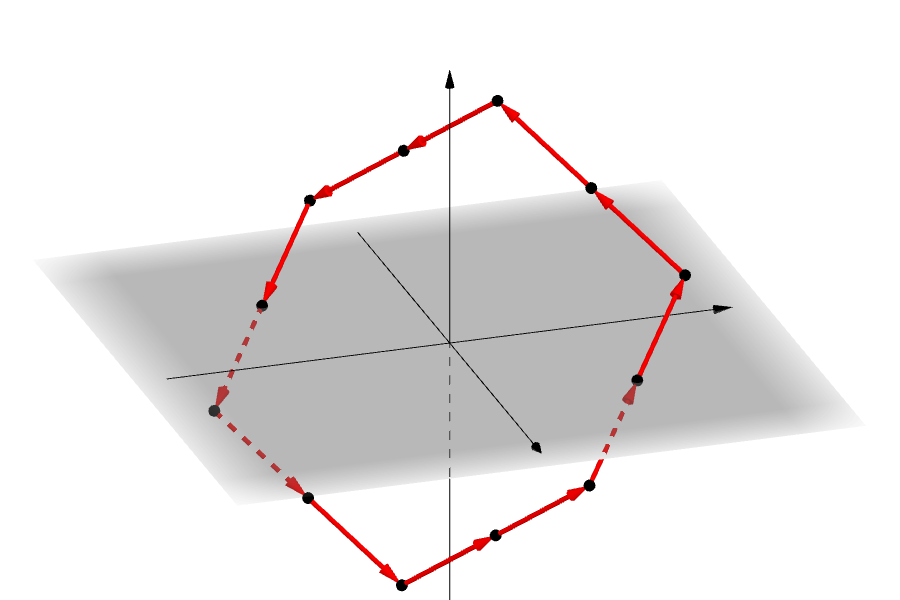}
\end{tabular}
\caption{Projections of the trajectory to $p$ and $q$ spaces.}
\label{Figure 3}
\end{figure}

% Рисунок 4

\begin{figure}[ht]\center
\begin{tabular}{cc}
\includegraphics[width=80mm]{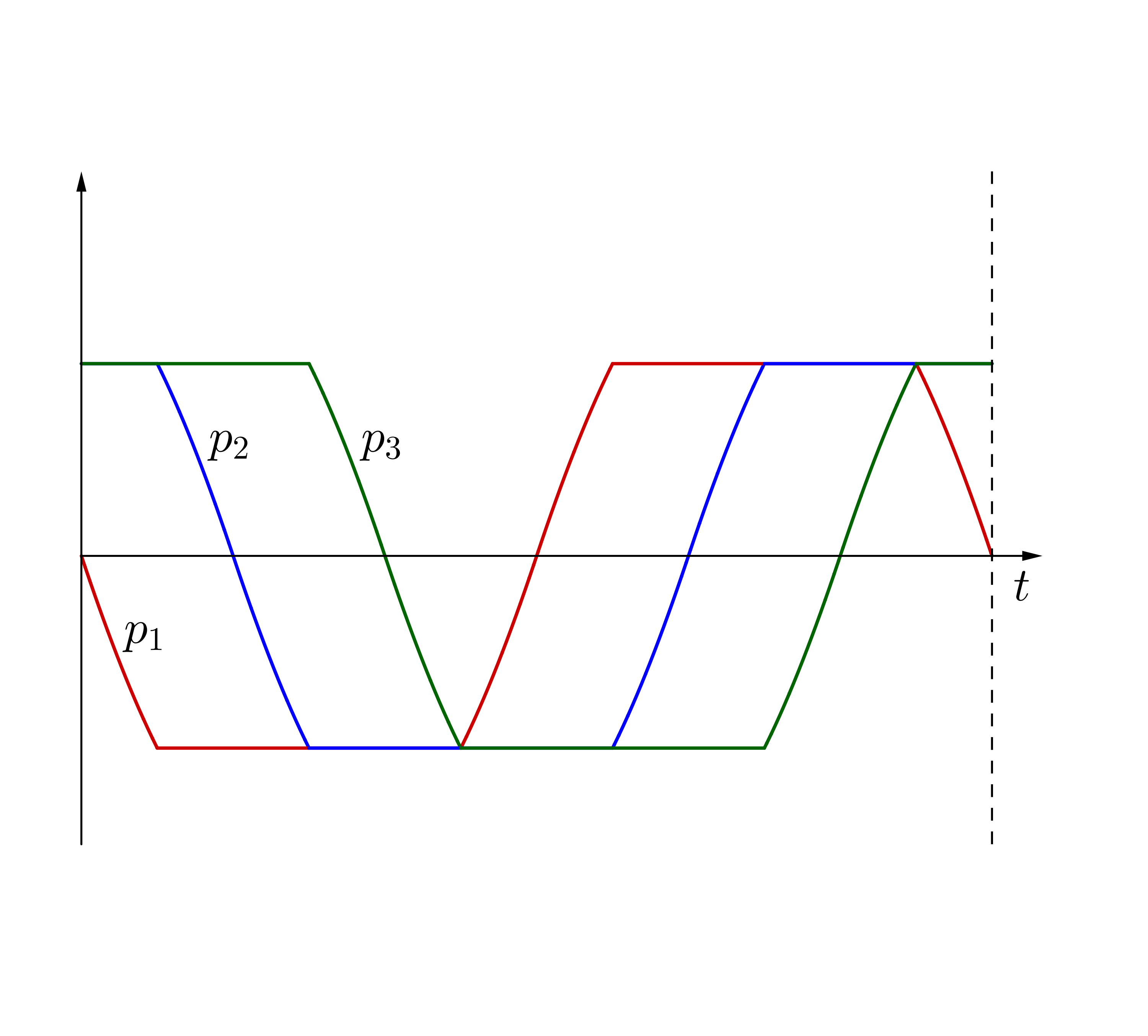}
&
\includegraphics[width=80mm]{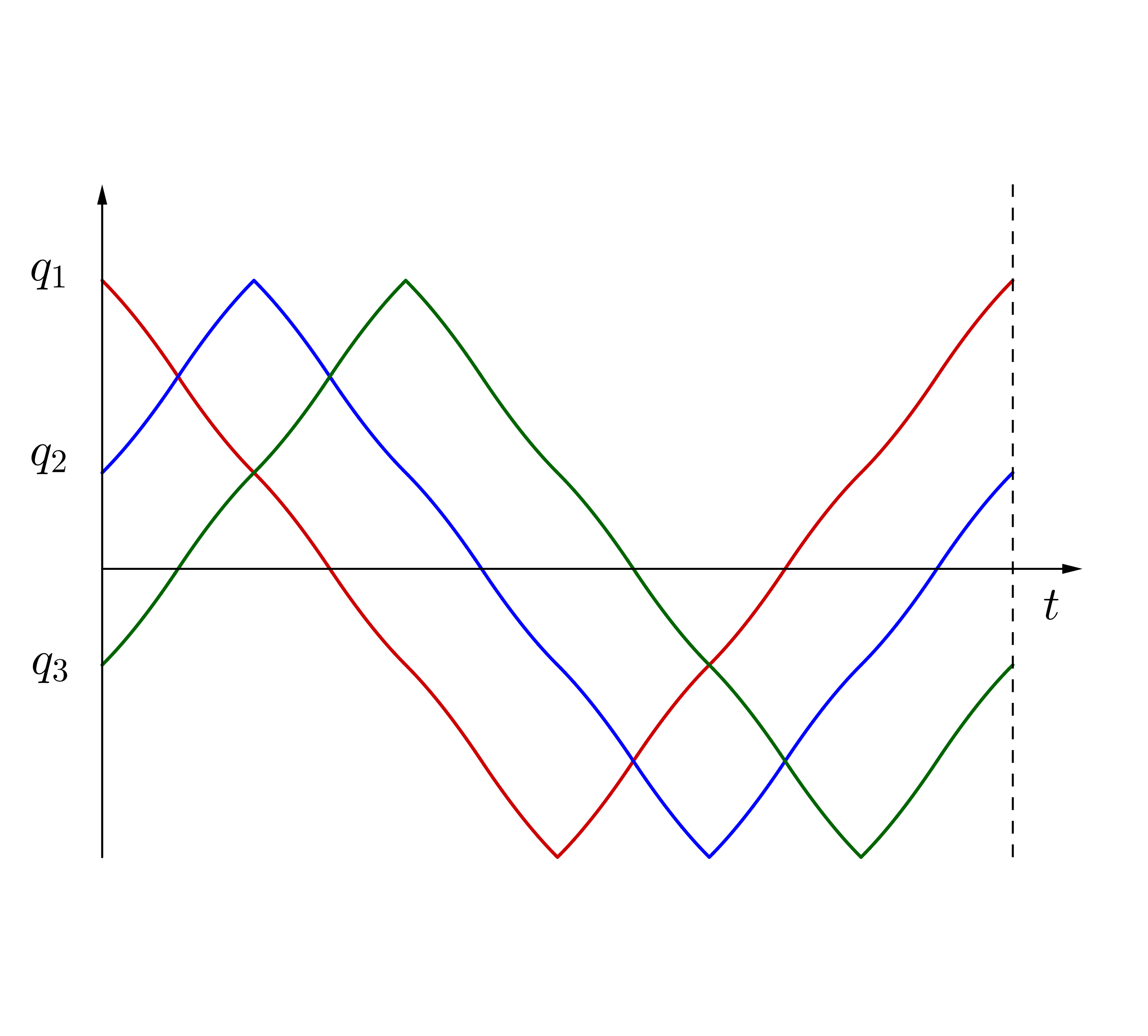}
\end{tabular}
\caption{Time evolution of $p$ and $q$ coordinates.}
\label{Figure 4}
\end{figure}

As in the case $n = 2$, let us introduce the angles $\alpha_i, \beta_i$:
\begin{eqnarray*}
 |q_1(0)| &=& \sin\alpha_1;\\
 |q_1(t_1)| = |q_2(t_1)| &=& \sin\beta_1;\\
 |q_2(t_2)| &=& \sin\alpha_2;\\
 |q_2(t_3)| = |q_3(t_3)| &=& \sin\beta_2;\\
 &\dots & \\
 |q_n(t_{4n-1})| &=& \sin\alpha_{2n};\\
 |q_n(t_{4n})| = |q_1(t_{4n})| &=& \sin\beta_{2n}.
\end{eqnarray*}
However, for our trajectory the angles only have two different values:
\[
\sin\alpha_i = \frac{n}{\sqrt{n^2 + (n-1)^2}} = \cos\beta_i,\quad \sin\beta_i = \frac{n - 1}{\sqrt{n^2 + (n-1)^2}} = \cos\alpha_i.
\]
Thus it is easy to find the period of the trajectory:
\[
T = \sum_{i=1}^{2n}\alpha_i - \sum_{i=1}^{2n}\beta_i = 2n(\alpha - \beta),
\] 
where $\alpha = \alpha_i, \beta = \beta_i$. Eventually, we have an explicit expression for $T$,
\[
T = 2n\arcsin{\frac{2n-1}{n^2 + (n-1)^2}},
\]
and we want to compare it to the value $4 \cdot\left(\frac{((\frac{n}{2})!)^2}{n!} \right)^{\frac{1}{n}}$ that arises from Viterbo's conjecture. In fact, we will prove Viterbo's conjecture for our body $\{H\le 1\}$ in the form:
\begin{equation}
    \label{ineq}
    4\cdot\left(\frac{((\frac{n}{2})!)^2}{n!} \right)^{\frac{1}{n}} \ge 2n\arcsin{\frac{2n-1}{n^2 + (n-1)^2}}.
\end{equation}
First, we get the following inequality 
\begin{equation}
    \label{stirling}
\frac{n!}{((\frac{n}{2})!)^2} \le \frac{2^{n}}{\sqrt{\pi\frac{n}{2}}}.
\end{equation}
Indeed, we can deduce by induction the expression 
\[
\frac{1}{2^n} \cdot \frac{n!}{((\frac{n}{2})!)^2} = \frac{2}{\pi} \int_0^{\frac{\pi}{2}}\cos{x}^n dx,
\] 
and write the inequality 
\[
\frac{2}{\pi} \int_0^{\frac{\pi}{2}}\cos{x}^n dx \le \frac{2}{\pi} \int_0^{\frac{\pi}{2}}\exp{\left(-\frac{n}{2}x^2\right)} dx,
\]
since for $x \in [0, \frac{\pi}{2}]$ we have $\cos{x} \le \exp{\left(-\frac{x^2}{2}\right)}$. We can also evaluate the right hand side 
\[
\frac{2}{\pi} \int_0^{\frac{\pi}{2}}\exp{\left(-\frac{n}{2}x^2\right)} dx < \frac{2}{\pi} \int_0^{\infty}\exp{\left(-\frac{n}{2}x^2\right)} dx = \frac{2}{\pi} \cdot \frac{\sqrt{\pi}}{2}\cdot \sqrt{\frac{2}{n}} = \sqrt{ \frac{2}{\pi n}}.
\]
 Hence, we prove (\ref{stirling}) and it can be rewritten in a more convenient way as
\[
4\cdot\left(\frac{((\frac{n}{2})!)^2}{n!} \right)^{\frac{1}{n}} \ge 2 \cdot \left(\frac{\pi n}{2}\right)^{\frac{1}{2n}}.
\]
In order to prove (\ref{ineq}) it is sufficient to check:
\[
2\cdot \left(\frac{\pi n}{2}\right)^{\frac{1}{2n}} \ge 2n\arcsin{\frac{2n-1}{n^2 + (n-1)^2}}
\Leftrightarrow
\sin\left(\left(\frac{\pi n}{2}\right)^{\frac{1}{2n}}\frac{1}{n}\right) \ge \frac{2n-1}{2n^2 - 2n +1}.
\]
Note that $\sin(x) \ge x - \frac{x^3}{6}$ when $x \ge 0$, therefore
\[
\sin\left(\left(\frac{\pi n}{2}\right)^{\frac{1}{2n}}\frac{1}{n}\right) \ge \left(\frac{\pi n}{2}\right)^{\frac{1}{2n}}\frac{1}{n} - \frac{1}{6} \cdot \left(\frac{\pi n}{2}\right)^{\frac{3}{2n}}\frac{1}{n^3}.
\]
On the other hand 
\[
n \cdot \frac{2n-1}{2n^2 - 2n +1} = \frac{2n^2 - n}{2n^2 - 2n + 1} = 1 + \frac{n - 1}{2n^2 - 2n + 1} \le 1 + \frac{1}{2n}.
\]
So it is sufficient to check the inequality:
\[
\left(\frac{\pi n}{2}\right)^{\frac{1}{2n}} - \frac{1}{6} \cdot \left(\frac{\pi n}{2}\right)^{\frac{3}{2n}}\frac{1}{n^2} \ge 1 + \frac{1}{2n}
\Leftrightarrow
\left(\frac{\pi}{4}\right)^{\frac{1}{2n}} \cdot (2n)^{\frac{1}{2n}} \ge 1 + \frac{1}{2n} + \frac{2}{3} \cdot \left(\frac{\pi}{4}\right)^{\frac{3}{2n}}\frac{(2n)^{\frac{3}{2n}}}{(2n)^2},
\]
or, after the variable change $x = 2n$,
\[
\left(\frac{\pi}{4}\right)^{\frac{1}{x}} \cdot x^{\frac{1}{x}} \ge 1 + \frac{1}{x} + \frac{2}{3} \cdot \left(\frac{\pi}{4}\right)^{\frac{3}{x}}\frac{x^{\frac{3}{x}}}{x^2}.
\]
We can estimate the left hand side
\[
\left(\frac{\pi}{4}\right)^{\frac{1}{x}} \cdot x^{\frac{1}{x}} = \exp\left(\left(\ln{x} - \ln\frac{4}{\pi}\right)\frac{1}{x}\right) \ge 1 + \left(\ln{x} - \ln\frac{4}{\pi}\right)\frac{1}{x},
\]
and reduce the problem to a simple inequality
\[
1 + \left(\ln{x} - \ln\frac{4}{\pi}\right)\frac{1}{x} \ge 1 + \frac{1}{x} + \frac{2}{3} \cdot \left(\frac{\pi}{4}\right)^{\frac{3}{x}}\frac{x^{\frac{3}{x}}}{x^2}
\Leftrightarrow
\ln{x} - \ln\frac{4}{\pi} \ge 1 + \frac{2}{3} \cdot \left(\frac{\pi}{4}\right)^{\frac{3}{x}}\frac{x^{\frac{3}{x}}}{x}.
\]
For $x \ge 10$ we have
\[
\ln{x} \ge \ln{10} > 2,3 > 2 + \ln{\frac{4}{\pi}} \ge 1 + \frac{2}{3} \cdot \left(\frac{\pi}{4}\right)^{\frac{3}{x}}\frac{x^{\frac{3}{x}}}{x} + \ln{\frac{4}{\pi}},
\]
and it remains to prove (\ref{ineq}) only for $n = 1, 2, 3, 4$:
\[
 4\cdot\left(\frac{((\frac{1}{2})!)^2}{1!} \right)^{\frac{1}{1}} = \pi = 2\arcsin{\frac{1}{1}},
 \]
\[
4\cdot\left(\frac{((\frac{2}{2})!)^2}{2!} \right)^{\frac{1}{2}} = 2\sqrt{2} > 2,6 > 4\arcsin{\frac{3}{5}},
\]
\[
4\cdot\left(\frac{((\frac{3}{2})!)^2}{3!} \right)^{\frac{1}{3}} = (6\pi)^\frac{1}{3} > 2,66 > 6\arcsin{\frac{5}{13}},
\]
\[
4\cdot\left(\frac{((\frac{4}{2})!)^2}{4!} \right)^{\frac{1}{4}} > 2,55 > 8\arcsin{\frac{7}{25}}.
\]
Thus we have proved \eqref{ineq} for all $n$, therefore this trajectory satisfies Viterbo's inequality and the conjecture holds for this Hamiltonian.
\end{proof}

\bibliography{../Bib/karasev}
\bibliographystyle{abbrv}
\end{document}